\documentclass[3p,11pt,authoryear]{elsarticle}

\makeatletter
\def\ps@pprintTitle{%
\let\@oddhead\@empty
\let\@evenhead\@empty
\def\@oddfoot{}%
\let\@evenfoot\@oddfoot}
\makeatother

\usepackage{amssymb}
\usepackage{amsthm}
\usepackage{lineno}  

\usepackage{enumerate,graphicx,mathtools,natbib,geometry,xcolor}
\usepackage[colorlinks]{hyperref}

\newtheorem{theorem}{Theorem}[section]

\newtheorem{corollary}[theorem]{Corollary}
\newtheorem{lemma}[theorem]{Lemma}
\newtheorem{remark}[theorem]{Remark}

\makeatletter \@addtoreset{equation}{section} \makeatother

\newcommand{\e}{\varepsilon}
\newcommand{\ii}{\mathrm{i}}
\newcommand{\N}{\mathbb N}

\newcommand{\C}{\mathbb C}

\newcommand{\rd}{{\rm d}}

\renewcommand{\Re}{\mathrm{Re}}
\renewcommand{\Im}{\mathrm{Im}}

\renewcommand{\P}{\mathbb{P}}

\newcommand{\E}{\mathbb{E}}

\newcommand{\OO}{\mathcal O}

\newcommand{\leqdef}{\vcentcolon=}

\begin{document}

\begin{frontmatter}

    \title{Discretization of the maximum for the derivatives of a\\ random model of \texorpdfstring{$\log|\zeta|$}{log|zeta|} on the critical line}

    \author[a3]{Louis-Pierre Arguin\fnref{fn3}}
    \author[a1]{Fr\'ed\'eric Ouimet\corref{cor1}\fnref{fn1}}
    \ead{ouimetfr@dms.umontreal.ca}

    \author[a2]{Christian Webb\fnref{fn2}}

    \address[a3]{Baruch College and Graduate Center (CUNY), New York, NY 10010, USA.}
    \address[a1]{California Institute of Technology, Pasadena, CA 91125, USA.}
    \address[a2]{Aalto University, Aalto, FI-00076, Finland.}

    \cortext[cor1]{Corresponding author}

    \fntext[fn3]{L.-P. Arguin is supported by NSF CAREER DMS-1653602.}
    \fntext[fn1]{F. Ouimet is supported by a postdoctoral fellowship from the NSERC (PDF) and a postdoctoral fellowship supplement from the FRQNT (B3X).}
    \fntext[fn2]{C. Webb is supported by the Academy of Finland grant 308123.}

    \begin{abstract}
        In this short note, we study a random field that approximates the real part of the logarithm of the Riemann zeta function on the critical line, introduced by \cite{arXiv:1304.0677,MR3619786}, and its derivatives of all orders.
        We show that the maximum of the random field, and more generally the maximum of its $j$-th derivative, varies on a $(\log T)^{\scriptscriptstyle -\frac{1}{2}(j+2)}$ scale, which improves and extends the main result in \cite{MR3906393} and makes further progress towards the open problem of the tightness of the recentered maximum.
        Our proof is also much simpler and shorter.

        \vspace{3mm}
        \noindent
        {\bf R\'esum\'e}

        \vspace{1.8mm}
        \noindent
        Dans cette courte note, nous \'etudions un champ al\'eatoire qui approxime la partie r\'eelle du logarithme de la fonction z\^eta de Riemann sur la ligne critique, introduit par \cite{arXiv:1304.0677,MR3619786}, et ses d\'eriv\'ees de tous les ordres.
        Nous montrons que le maximum du champ al\'eatoire, et plus g\'en\'eralement le maximum de sa $j$-i\`eme d\'eriv\'ee, varie sur une \'echelle $(\log T)^{\scriptscriptstyle -\frac{1}{2}(j+2)}$, ce qui am\'eliore et \'etend le r\'esultat principal dans \cite{MR3906393} et nous fait progresser concernant le probl\`eme ouvert de la tension du maximum recentr\'e.
        Notre preuve est aussi beaucoup plus simple et courte.
    \end{abstract}


\end{frontmatter}

\section{Model and background}

    Let $(U_p, \, p ~\text{primes})$ be an i.i.d.\ sequence of uniform random variables on the unit circle in $\C$.
    The random field of interest is
    \begin{align}\label{def:X}
        X_T(h) \leqdef \sum_{p \leq T} \frac{\text{Re}(U_p \, p^{-i h})}{p^{1/2}}, \quad h\in [0,1].
    \end{align}
    (A sum over the variable $p$ always denotes a sum over primes.)
    \cite{arXiv:1304.0677} showed that $(X_{\scriptscriptstyle T}(h), \, h\in [0,1])$ is a good model for the large values of $(\log |\zeta(\tfrac{1}{2} + \ii T + \ii h)|, \, h\in [0,1])$ when $T$ is large, if we assume the Riemann hypothesis.
    The second order of the maximum was shown in \cite{MR3619786}, but the tightness of the recentered maximum of $X_{\scriptscriptstyle T}$ is still open.
    This is the motivation behind this paper.
    Our result shows that the maximum of $X_{\scriptscriptstyle T}$ can be discretized to a number of points, $\log T$, that coincides with the number of leaves in the approximate branching structure underlying $\log|\zeta|$.
    This is a non-trivial improvement over the main result in \cite{MR3906393}. 
    Our method yields similar discretizations for all the derivatives of $X_{\scriptscriptstyle T}$.

    \vspace{3mm}
    For various asymptotic results of interest on the extreme values of the model in \eqref{def:X}, see \cite{arXiv:1304.0677,MR3619786,arXiv:1706.08462,arXiv:1906.08573,MR3906393,MR3841407,Ouimet2019phd,arXiv:1604.08378,arXiv:1609.00027}.
    For asymptotic results on the maximum of the Riemann zeta function on the critical line, we refer the reader to \cite{MR3851835,ABBRS_2019,arXiv:1901.04061,arXiv:1906.05783,MR3662441,arXiv:1804.01629} and references therein.
    Related conjectures can be found in \cite{MR2350784,FyodorovHiaryKeating2012,MR3151088}.

\section{Result}

    Below, we work with the increments of the field $X_{\scriptscriptstyle T}$.
    For $-1 \leq r \leq k$, let
    \begin{equation}\label{def:G.r.k}
        X_{r,k}(h) = \sum_{2^r < \log p \leq 2^k} \hspace{-2mm} \frac{\Re(U_p \, p^{-\ii h})}{p^{1/2}}, \quad h\in [0,1].
    \end{equation}
    The $j$-th derivative of $X_{\scriptscriptstyle r,k}$ is
    \begin{equation*}
        X_{r,k}^{(j)}(h) \leqdef \frac{\rd^j}{\rd h^j} X_{r,k}(h) =
        \begin{cases}
            \sum_{2^r < \log p \leq 2^k} (-1)^{j/2} \, \frac{(\log p)^j}{p^{1/2}} \, \Re(U_p p^{-\ii h}), &\hspace{-1mm}\mbox{if } j ~\text{is even}, \\[2mm]
            \sum_{2^r < \log p \leq 2^k} (-1)^{(j-1)/2} \, \frac{(\log p)^j}{p^{1/2}} \, \Im(U_p p^{-\ii h}), &\hspace{-1mm}\mbox{if } j ~\text{is odd},
        \end{cases}
    \end{equation*}

    \vspace{0.5mm}
    \noindent
    which can be seen as a toy model for the real part of the $j$-th logarithmic derivative of the Riemann zeta function on the critical line.

    \vspace{4mm}
     The theorem below says that if we want to find the asymptotics of the maximum of $X_{\scriptscriptstyle r,k}^{\scriptscriptstyle (j)}$ up to a constant, we can restrict the maximum to a discrete set containing $\OO(2^{\scriptscriptstyle \frac{1}{2}(j+2)k})$ equidistant points.

    \vspace{1mm}
    \begin{theorem}[Discretization]\label{thm:discretization}
        Fix $j\in \N_0$ and let $\mathcal{H}_k \leqdef \alpha 2^{-\frac{1}{2}(j+2)k} \hspace{0.3mm} \N_0 \cap [0,1]$ for $\alpha > 0$.
        Then, for all $\e, K > 0$, there exists $\alpha = \alpha(j,\e,K) > 0$ small enough that
        \begin{equation}\label{eq:thm:discretization}
            \P\Big(\big|\max_{h\in [0,1]} X_{r,k}^{(j)}(h) - \max_{h\in \mathcal{H}_k} X_{r,k}^{(j)}(h)\big| > K\Big) < \e.
        \end{equation}
        Similarly, if $\alpha > 0$ is small enough with respect to $j$ and $\e$, there exists $K = K(j,\e,\alpha) > 0$ large enough that \eqref{eq:thm:discretization} holds.
    \end{theorem}

    \vspace{1mm}
    \begin{remark}
        When $r = -1$ and $2^k = \log T$, $X_{\scriptscriptstyle r,k}$ is the full model $X_{\scriptscriptstyle T}$, and $|\mathcal{H}_k| = \OO(\log T)$.
        Theorem \ref{thm:discretization} improves the main result in \cite{MR3906393}, where only the case $j = 0$ was treated and the discrete set had $\OO(\sqrt{\log \log T} \log T)$ points instead.
        The proof rested on estimates of the joint Laplace transform of $X_{\scriptscriptstyle r,k}^{\scriptscriptstyle (1)}$ and continuity estimates derived from a chaining argument of \cite{MR3619786}.
        Our proof here is much simpler, much shorter, and also applies for the derivatives of all orders.
        The main idea is to control the maximum of $X_{\scriptscriptstyle r,k}^{\scriptscriptstyle (j+1)}$ around the point $h^{\star}$ where $X_{\scriptscriptstyle r,k}^{\scriptscriptstyle (j)}$ maximizes by applying a mean value theorem to $X_{\scriptscriptstyle r,k}^{\scriptscriptstyle (j+1)}$ followed by Jensen's inequality and a bound on the variance of $X_{\scriptscriptstyle r,k}^{\scriptscriptstyle (j+2)}$ using prime number theorem estimates.
    \end{remark}

\section{Proof of Theorem \ref{thm:discretization}}

        Fix $j\in \N_0$ and let
        \begin{equation}\label{eq:h.star.x.star}
            \begin{aligned}
                &h^{\star} \leqdef \text{argmax}_{h\in [0,1]} X_{r,k}^{(j)}(h), \\
                &x^{\star} \leqdef \text{argmax}_{x\in [0,1] \, : \, |x - h^{\star}| \leq \alpha 2^{-\frac{1}{2}(j+2)k}} X_{r,k}^{(j+1)}(x).
            \end{aligned}
        \end{equation}
        (Note that $h^{\star}, x^{\star}\in (0,1)$ with probability $1$.)
        By the mean value theorem, the fact that $X_{\scriptscriptstyle r,k}^{\scriptscriptstyle (j+1)}(h^{\star}) = 0$, and by Jensen's inequality, we have
        \begin{align}
            \big|X_{r,k}^{(j+1)}(x^{\star})\big|^2
            = \Big|\int_{h^{\star}}^{x^{\star}} \hspace{-1.5mm} X_{r,k}^{(j+2)}(x) \rd x\Big|^2
            \leq |x^{\star} - h^{\star}| \int_{h^{\star}}^{x^{\star}} \big|X_{r,k}^{(j+2)}(x)\big|^2 \rd x,
        \end{align}
        which implies, using $(\Re \, z)^2 = (z + \overline{z})^2/4$, $(\Im \, z)^2 = (z - \overline{z})^2/(-4)$, $\E[U_p^2] = \E[\overline{U_p}^2] = 0$ and the independence of the $U_p$'s,
        \begin{align}\label{eq:variance.bound}
            \E\Big[\big|X_{r,k}^{(j+1)}(x^{\star})\big|^2\Big]
            &\leq \alpha^2 2^{-(j+2)k} \max_x \E\Big[\big|X_{r,k}^{(j+2)}(x)\big|^2\Big] \notag \\[2mm]
            &= \alpha^2 2^{-(j+2)k} \hspace{1mm}\max_x\hspace{-1mm} \sum_{2^r < \log p \leq 2^k} \hspace{-2mm} \frac{(\log p)^{2(j+2)}}{p} \left\{\frac{0 \cdot p^{-2\ii x} \pm 2 \, \E[U_p \overline{U_p}] + 0 \cdot p^{2\ii x}}{\pm 4}\right\} \notag \\[1mm]
            &= \alpha^2 2^{-(j+2)k} \sum_{2^r < \log p \leq 2^k} \hspace{-2mm} \frac{(\log p)^{2(j+2)}}{2p}.
        \end{align}
        (Here, $\pm$ means $+$ when $j$ is even and $-$ when $j$ is odd.)
        By the prime number theorem estimates in Lemma \ref{lem:PNT.estimates}, the last sum of primes is bounded above by $(2^{\scriptscriptstyle 2(j+2)k} - 2^{\scriptscriptstyle 2(j+2)r}) / (4(j+2)) + D_j$, where $D_j > 0$ is a constant that only depends on $j$.
        Hence, by Chebyshev's inequality, we can take $A = A(j,\e) > 0$ large enough that
        \begin{equation}\label{eq:prop:discretization.end.2.1}
            \P\Big(X_{r,k}^{(j+1)}(x^{\star}) > A \alpha 2^{\frac{1}{2}(j+2)k}\Big) < \e.
        \end{equation}

        Now, on the complementary event $\{0 \leq X_{\scriptscriptstyle r,k}^{\scriptscriptstyle (j+1)}(x^{\star}) \leq A \alpha 2^{\scriptscriptstyle \frac{1}{2}(j+2)k}\}$, apply the mean value theorem for $\exp(X_{\scriptscriptstyle r,k}^{\scriptscriptstyle (j)}(h^{\star}))$ around the point $h\in \mathcal{H}_k$ that is closest on the left-hand side of $h^{\star}$ (this choice implies $0 \leq h^{\star} - h \leq \alpha 2^{\scriptscriptstyle -\frac{1}{2}(j+2)k}$), then
        \begin{align}
            0 \leq e^{X_{r,k}^{(j)}(h^{\star})} - e^{X_{r,k}^{(j)}(h)}
            &= \int_h^{h^{\star}} X_{r,k}^{(j+1)}(x) \, e^{X_{r,k}^{(j)}(x)} \rd x \notag \\[0.5mm]
            &\leq \alpha 2^{\scriptscriptstyle -\frac{1}{2}(j+2)k} \cdot X_{r,k}^{(j+1)}(x^{\star}) \, e^{X_{r,k}^{(j)}(h^{\star})} \notag \\[0.5mm]
            &\leq A \alpha^2 \cdot e^{X_{r,k}^{(j)}(h^{\star})}.
        \end{align}
        For any given $K > 0$, choose $\alpha = \alpha(j,\e,K) > 0$ small enough that $A \alpha^2 \leq 1 - e^{\scriptscriptstyle -K}$.
        We find that
        \begin{equation}
            e^{X_{r,k}^{(j)}(h)} \geq e^{X_{r,k}^{(j)}(h^{\star})} \, e^{-K} ~\text{on the event $\{0 \leq X_{r,k}^{\scriptscriptstyle (j+1)}(x^{\star}) \leq A \alpha 2^{\frac{1}{2}(j+2)k}\}$},
        \end{equation}
        and thus
        \begin{equation}\label{eq:prop:discretization.end.2.2}
            \P\Big(X_{r,k}^{(j)}(h^{\star}) \geq \max_{h\in \mathcal{H}_k} X_{r,k}^{(j)}(h) \geq X_{r,k}^{(j)}(h^{\star}) - K\Big) \geq 1 - \e,
        \end{equation}
        by \eqref{eq:prop:discretization.end.2.1}.
        This ends the proof.

    \vspace{1mm}
    \begin{remark}
        It is straightforward to verify that the proof of Theorem \ref{thm:discretization} carries over to the Gaussian model
        \vspace{-2mm}
        \begin{equation}
            G_{r,k}(h) = \sum_{2^r < \log p \leq 2^k} \hspace{-2mm} \frac{\Re(Z_p \, p^{-\ii h})}{p^{1/2}}, \quad h\in [0,1],
        \end{equation}
        from \cite{arXiv:1604.08378,arXiv:1609.00027}, where $Z_p = W_p^{\scriptscriptstyle (1)} + \ii W_p^{\scriptscriptstyle (2)}$ \hspace{-1.3mm} and \hspace{-0.3mm} $W_p^{\scriptscriptstyle (1)}\hspace{-1mm},\hspace{0.5mm} W_p^{\scriptscriptstyle (2)} \sim \mathcal{N}(0,1/2)$ are i.i.d. This is because the variance bound in \eqref{eq:variance.bound} is also valid if we replace $X$ by $G$.
    \end{remark}

\section{Tool}

    The estimates below are simple consequences of the prime number theorem and integration by parts.

    \begin{lemma}[Lemma A.1 in \cite{MR3906393}]\label{lem:PNT.estimates}
        Let $j\in \N^*$ and $1 \leq P < Q$, then
        \begin{equation}
            \bigg|\sum_{P < p \leq Q} \hspace{-1mm}\frac{(\log p)^j}{p} - \left(\frac{(\log Q)^j}{j} - \frac{(\log P)^j}{j}\right)\bigg| \leq D_j,
        \end{equation}
        for some constant $D_j > 0$ that only depends on $j$.
    \end{lemma}

%
%

\bibliographystyle{authordate1}
\bibliography{Ouimet_PhD_bib}

\end{document}